\documentclass{article}
\RequirePackage{natbib}
\RequirePackage{amsthm,amsmath,latexsym,amssymb}

\numberwithin{equation}{section}
\theoremstyle{plain}
\newtheorem{thm}{Theorem}[section]
\newtheorem{lemma}{Lemma}[section]

\begin{document}

\title{Parameter estimation of high-dimensional linear differential equations}


\author{Heng Lian\\Nanyang Technological University\\
Division of Mathematical Sciences\\
Nanyang Technological University\\
Singapore, 637371\\
}
\maketitle

\begin{abstract}
We study the problem of estimating the coefficients in linear ordinary differential equations (ODE's) with a diverging number of variables when the solutions are observed with noise. The solution trajectories are first smoothed with local polynomial regression and the coefficients are estimated with nonconcave penalty proposed by \cite{fan01}. Under some regularity and sparsity conditions, we show the procedure can correctly identifies nonzero coefficients with probability converging to one and the estimators for nonzero coefficients have the same asymptotic normal distribution as they would have when the zero coefficients are known and the same two-step procedure is used. Our asymptotic results are valid under the misspecified case where linear ODE's are only used as an approximation to nonlinear ODE's, and the estimates will converge to the coefficients of the best approximating linear system. From our results, when the solution trajectories of the ODE's are sufficiently smooth, the parametric $\sqrt{n}$ rate is achieved even though nonparametric regression estimator is used in the first step of the procedure. The performance of the two-step procedure is illustrated by a simulation study as well as an application to yeast cell-cycle data.
\end{abstract}

\section{Introduction}
Ordinary differential equations are widely used to describe systems in physics, chemistry and biology. Many such systems can be described by the initial value problem
\begin{equation}\label{eqn:nonlinear}
\left\{\begin{array}{lll}
m'=F(m,\mathbf{\theta})\\
m(0)=m_0
\end{array}\right.
\end{equation}
where $m=(m_1,\ldots,m_p)^T$ represents the state of the system.
When some simple regularity conditions on the smoothness of $F$ are imposed, there exists a unique solution of the nonlinear ODE, at least in a small neighborhood of zero. Analytical insolvability of nonlinear equations necessitates numerical methods to find the solution for the initial value problem. On the other hand, the statisticians are concerned with the estimation of the parameters $\theta$ ($F$ is assumed known) given noisy solutions $Y_{ij}=m_j(X_{i})+\epsilon_{ij}, j=1,\ldots,p,$ observed at time points $X_1, X_2, \ldots, X_{n_i}$. This problem has been investigated by many authors. There exist roughly two classes of approaches. The first approach uses classical parametric inference, such as the nonlinear least square estimator or maximum likelihood estimator \cite{biegler86}. Optimization usually involves an iterative process. Starting from fixed initial values $m_0$, it finds the solution of (\ref{eqn:nonlinear}) using numerical methods such as Euler or Runge-Kutta based on the current parameter estimates. Similarly, inferences in \cite{gelman96} is based  on the Bayesian principle and the observations are modeled by, for example, $Y_{ij}\sim N(m_j(X_i),\sigma^2)$, which also requires numerically solving the of ODE's. Besides, MCMC should be used for posterior computation. If the initial values are unknown, they should also be considered as parameters and optimized together with $\theta$ in (\ref{eqn:nonlinear}).
The second family of approaches which is closely related to ours is to directly minimize deviation of $m'$ from $f(m,\theta)$. \cite{varah82} proposed a two-step method, in which $m$ is first estimated from noisy data using cubic splines and then $\int ||m'-f(m,\theta)||^2$ is minimized with respect to $\theta$. In this approach, numerical solution of ODE is not required and  unknown initial values do not add to computational burden. \cite{ramsay07} extends this approach using a single step method and optimizes the criterion that represents a trade-off between the fidelity to ODE and the data fit. The computations for the single-step approach are more involved than the two-step approach. 

We consider the simpler linear system of ODE's  with a large number of coefficient parameters.
\begin{equation}\label{eqn:linear}
\left\{\begin{array}{l}
m'=Am\\
m(0)=m_0
\end{array}\right.
\end{equation}
where $m=(m_1,\ldots,m_p)$ and A is the $p\times p$ coefficient matrix. Note that for simplicity we do not include a constant term in the system. The solution of this system of ODE's is well known and is determined by the spectrum of the matrix $A$, although a full discussion considering all possible cases is complicated when $p$ is large. We also regard the linear ODE's as an approximation to the truth so that $\{m_j\}$ is not necessarily the solution of (\ref{eqn:linear}). In high dimensions, linear approximations to nonlinear ODE's make more sense since specification of nonlinearity is a much complicated matter. In this case, $m_j$ is not necessarily an analytical function as in the linear system, but we will still assume it is sufficiently smooth later. 

If we  use a nonparametric estimator for the solution as well as its derivative, denoted by $\hat{m}$ and $\hat{m'}$, the fidelity to the ODE's can be assessed by 
\begin{equation}\label{eqn:lsq}
\int ||m'-Am||^2\,dx
\end{equation}
where $||\cdot ||$ denotes the Euclidean norm. Obviously each of the $p$ equations can be considered and fitted separately and the estimation problem is $p$ dimensional instead of $p^2$ when all parameters are considered together. 

Unlike the standard linear regression, even when $p$ is large, there still exists a unique solution for the least square problem (\ref{eqn:lsq}) under mild assumptions. But when $p$ is large, either because of a priori beliefs on the sparsity of the matrix or due to consideration of interpretability of the resulting model, regularized or penalized method is needed. For standard linear regression, Lasso \cite{tibshirani96} is probably the most popular method that uses the $L_1$ penalty
\begin{equation*}
||y-X\beta||^2+\lambda\sum_{i=1}^p|\beta_j|.
\end{equation*}
The $L_1$ penalty will force some of the coefficients to be equal to zero. Compared to traditional model selection method using information criteria, Lasso is continuous and thus more stable. More systematic theoretical studies on Lasso appeared later. \cite{greenshtein04} showed that Lasso is consistent for prediction, a property that was called persistency. Several authors \cite{meinshausen06,zhaoyu06} have shown that Lasso is in general not consistent for model selection unless some nontrivial conditions on the covariates are satisfied. Even when those conditions are satisfied, the efficiency of the estimator is compromised when one insists on variable selection consistency since the coefficients are over-shrinked. To address these shortcomings of Lasso, \cite{fan01} proposed the smoothly clipped absolute deviation (SCAD) penalty which is motivated by taking into account several desired properties of the estimator like continuity, asymptotic unbiasedness, etc. They also show that the resulting estimator possesses the oracle property, i.e. it is consistent for variable selection and behaves the same as  when the zero coefficients are known in advance. These results are extended to the case with a diverging number of covariates in \cite{fan04}. \cite{zou06} proposed adaptive lasso in the fixed $p$ case using a weighted $L_1$ penalty with weights determined by an initial estimator and similar oracle property followed. The idea behind the adaptive lasso is to assign higher penalty for zero coefficients and lower penalty for larger coefficients. \cite{huangma06} studied the adaptive lasso with a diverging number of parameters and proposed using marginal regression as the initial estimator under partial orthogonality assumption. Also in the high dimensional case, \cite{huangma08} showed similar oracle properties for the estimator with $L_\gamma$ penalty when $0<\gamma<1$. 

In this paper, we study the asymptotic properties within the framework of sparse linear ODE using the SCAD penalty. Since the $p$ equations are considered separately, we assume without loss of generality that we only want to estimate the first equation by minimizing
\begin{equation*}
\int (m'_1(x)-\beta^T m(x))^2\,dx+\sum_{j=1}^p p_{\lambda}(|\beta_j|)
\end{equation*}
using some estimator for $m$ and its derivatives. The SCAD penalty is defined by 
\begin{equation*}
p'_\lambda(\theta)=\lambda\{I(\theta\le\lambda)+\frac{(a\lambda-\theta)_+}{(a-1)\lambda}I(\theta>\lambda)\}~~\mbox{for some } a>2 \mbox{ and } \theta>0.
\end{equation*}
Other penalties like adaptive lasso and $L_\gamma$ discussed above will lead to similar asymptotic results, although  initial estimators are required in those cases.

The rest of the paper is organized as follows. Section 2 presents our two-step procedure using local polynomial regression as the solution estimator. The asymptotic property of the estimator is discussed. Under regularity conditions, we show the estimated coefficients still have parametric convergence rates even with nonparametric regression estimates from the first step plugged in. The oracle property is shown. In section 3, we conduct a simulation study to assess the finite sample performance and use a real dataset as an illustration of the procedure. We make some concluding remarks in section 4. The proofs are collected in section 5.

\section{Two-step estimator and its asymptotic properties}
\subsection{Two-step estimation}
For a general nonlinear system of ODE's (\ref{eqn:nonlinear}), we observe its solution with additive noise
\begin{equation*}
Y_{ij}=m_j(X_i)+\epsilon_{ij}, ~i=1,\ldots,n,~j=1,\ldots,p_n.
\end{equation*}
For simplicity of notation and proof, we assume the $n$ observation time points $\{X_i\}_{i=1}^n$ are i.i.d. from a uniform distribution on the interval $(0,1)$. The observation times for all $p_n$ variables are assumed to be the same. Although our estimator certainly works with different observation times for different variable $m_j$, the above assumption of identical time points makes the proof more transparent. Note that we consider the case where the number of variables diverges with the number of observations for each variable.

Although the observed noisy solution comes from possibly nonlinear ODE's, we use a linear system as an approximation for modeling. The true parameters (more precisely, the best approximating parameters) is defined to be
\begin{equation}\label{eqn:minA}
A_0=\arg\min_A \int_0^1||m'(x)-Am(x)||^2w(x)dx
\end{equation}
where $m=(m_1, \ldots, m_{p_n})$ and $w(\cdot)$ is a pre-determined nonnegative weight function. We assume that a unique minimizer for (\ref{eqn:minA}) exists. Since the minimum is obviously independently defined for each row of $A$, we only focus on the first row and denote it by $\beta_0$. Let $\beta_0=(\beta_{10}^T,\beta_{20}^T)^T$ with $\beta_{20}=0$, where $\beta_{10}$ is a vector of length $k_n$ and $\beta_{20}$ is a vector of length $p_n-k_n$. This is the usual sparsity assumption used in various papers on high-dimensional penalized regression.

When given only the noisy data $Y_{ij}$, we first estimate $m_j$ using nonparametric regression. In this paper, we use the local polynomial estimator \cite{fangijbels03}. 
In local polynomial regression, for a smooth function $m(x)$ with noisy observations $Y_i=m(X_i)+\epsilon_i, i=1,\ldots,n$, we model $m(x)$ around some point $x_0$ by
\begin{equation*}
m(x)\approx\sum_{d=0}^s \frac{m^{(d)}(x)}{d!}(x-x_0)^d
\end{equation*}
where $m^{(d)}$ is the $d$-th derivative of $m$.

This motivated the minimization of the following objective function
\begin{equation}\label{eqn:lpe}
\min_{\alpha}\sum_{i=1}^n(Y_i-\sum_{d=0}^s\alpha_d(X_i-x)^d)^2K(\frac{X_i-x}{h})
\end{equation}
where a kernel function $K(\cdot)$ with bandwidth $h$ is used for localization. Let $\hat{\alpha}=(\hat{\alpha}_1, \ldots, \hat{\alpha}_d)$ be the solution to the problem (\ref{eqn:lpe}). The local polynomial estimator for $m^{(d)}(x_0)$ is $\hat{m}^{(d)}(x_0)=d!\hat{\alpha}_d$. Note in this paper, even though we use the notation $\hat{m}^{(d)}(\cdot)$ to denote the estimator of the derivatives, it is different from the derivative of $\hat{m}(\cdot)$.

Denote by $X$ the design matrix
\begin{eqnarray*}
X=\left(\begin{array}{cccc}
1&(X_1-x_0)&\cdots&(X_1-x_0)^s\\
\vdots&\vdots&&\vdots\\
1&(X_n-x_0)&\cdots&(X_n-x_0)^s
\end{array}\right),
\end{eqnarray*}
and 
\begin{equation*}
y=\left(\begin{array}{c}
Y_1\\
\vdots\\
Y_n
\end{array}\right),
\end{equation*}
the problem (\ref{eqn:lpe}) can be written as 
\begin{equation*}
\min_\alpha (y-X\alpha)^TW(y-X\alpha).
\end{equation*}
where $W$ is the diagonal matrix with $K(\frac{X_i-x_0}{h})$ for the $i$-th diagonal element. The solution can be written in a closed form 
\begin{equation*}
\hat{\alpha}=(X^TWX)^{-1}X^TWy
\end{equation*}
Some algebra shows that $\hat{\alpha}_d$ can also be written as
\begin{equation*}
\hat{\alpha}_d(x_0)=\sum_{i=1}^nW_d(\frac{X_i-x_0}{h})Y_i
\end{equation*}
for some weight functions $W_d(\cdot)$ depending on both $x_0$ and $X_i$.

After applying local polynomial regression to observations $Y_{ij}, i=1,\ldots, n$ for each $j, j=1, \ldots, p_n$, we estimate the coefficients $\beta$ by minimizing the penalized least square objective function
\begin{eqnarray}
S(\beta)&:=&\int_0^1 (\hat{m}_1(x)-\beta^T m(x))^2 w(x)dx + \sum_jp_{\lambda_n}(|\beta_j|)\label{eqn:obj}\\
\hat{\beta}&=&\arg\min_\beta S(\beta)\nonumber
\end{eqnarray}
where $\lambda_n$ is the smoothing parameter for the SCAD penalty $p_\lambda(\cdot)$. From the form of the objective function, $\hat{\beta}$ only depends on the nonparametric estimators through the functionals $\int \hat{m}_i(x)\hat{m}_j(x)w(x)dx$ and $\int \hat{m}_i(x)\hat{m}'_1(x)w(x)dx, i,j=1,\ldots, p_n$. Similar functionals are studied by different authors in the context of nonparametric density estimation or regression. For estimation of a density, say $f$, \cite{hall87} and \cite{bickel88} discuss estimation of $\int [f^{(d)}(x)]^2 dx$ using kernel estimator. \cite{laurent96,laurent97} uses series projection to estimate functionals of more general forms. In the context of regression, \cite{doksum95} gives estimator of $\int [m(x)]^2dx$ using kernel regression, and \cite{huangfan99} investigated the estimation of $\int [m^{(d)}(x)]^2 dx$ using local polynomial regression. 

\subsection{Asymptotic properties}
Before we present the first result, we need some notations. Denote the $p_n\times (p_n+1)$ matrix of integral functionals
\begin{eqnarray}\label{eqn:M}
&&M=\left[\begin{array}{ccccc}
\int m_1'm_1w&\int m_1m_1w&\int m_1m_2w&\ldots &\int m_1m_{p_n}w\\
\int m_1'm_2w&\int m_2m_1w&\int m_2m_2w&\ldots &\int m_2m_{p_n}w\\
\vdots       &\vdots       &\vdots       &\vdots       &\vdots\\
\int m_1'm_{p_n}w&\int m_{p_n}m_1w&\int m_{p_n}m_2w&\ldots &\int m_{p_n}m_{p_n}w\\
\end{array}\right] .
\end{eqnarray}
The corresponding matrix with local polynomial estimators plugged in is denoted by $\hat{M}$. For any $p\times (p+1)$ matrix $C$, let 
\begin{equation*}
vech(C)=(c_{11},c_{21},\ldots,c_{p1},c_{12},c_{22},\ldots,c_{p2},c_{23},\ldots c_{p+1,p})^T
\end{equation*}
be its vectorized version in $p(p+3)/2$ dimensions. Thus $vech(M)$ and $vech(\hat{M})$ contains all the nonrepetitive elements in the two matrices. Let $vec(C)$ be the usual vectorization of matrix $C$ in $p(p+1)$ dimensions. Obviously, there exists a binary matrix $\Phi_p$ such that $vec(C)=\Phi_p vech(C)$. From (\ref{eqn:M}), we can write the matrix of integral functional as $M=[b,Q]$, where $b$ is a $p_n$ dimensional vector consisting of functionals of the form $\int m_1'm_iw$ and $Q$ is the $p_n\times p_n$ matrix containing functionals of the form $\int m_im_jw$. Similarly we write $\hat{M}=[\hat{b},\hat{Q}]$.

Now we can state the regularity conditions for the consistency and oracle property of the SCAD penalized estimator.
\begin{enumerate}
\item[(A)] The kernel $K$ is a continuous bounded symmetric density function supported on $[-1,1]$.
\item[(B)] The true solution of nonlinear ODE's, $m_j, i=1,\ldots,p_n,$ is three times continuously differentiable, and local polynomial estimator used is of order $s=3$.
\item[(C)] The weight function $w$ is bounded and nonnegative, with $w^{(i)}(0)=w^{(i)}(1)=0, i=0,1,2$.
\item[(D)] The errors $\epsilon=(\epsilon_1^T,\ldots,\epsilon_{p_n}^T)^T$, with $\epsilon_j=(\epsilon_{1j}, \ldots, \epsilon_{nj})^T$ denoting the noises associated with $m_j$, are independent and identically distributed as $N(0,\sigma^2)$.
\item[(E)] $nh^6\rightarrow \infty, nh^4\rightarrow 0, \sqrt{n}p_n(h^3+\frac{1}{nh^2})\rightarrow 0$.
\item[(F)] The eigenvalues of the matrix $Q$ satisfies
	\begin{equation*}
		0<\rho_{1n}\le\lambda_{\min}(Q)\le\lambda_{\max}(Q)\le\rho_{2n}.
	\end{equation*}
\item[(G)] $\frac{p_n}{\sqrt{n}}=o(\rho_{1n})=o(\rho_{2n}), \lambda_n\rightarrow 0, \frac{p_n^2}{\sqrt{n}\rho_{1n}\lambda_n}\rightarrow 0, \frac{\rho_{2n}p_n^2}{\sqrt{n}\rho_{1n}\lambda_n}\rightarrow 0.$
\item[(H)] 
The nonzero coefficients, $\beta_{10}=(\beta_{01},\beta_{02},\ldots,\beta_{0k_n})^T$ satisfy 
\begin{equation*}
\max_{1\le j\le k_n}|\beta_{0j}|\le C, \mbox{ for some constant $C$ independent of $n$.}
\end{equation*}

\item[(\,I\,)]
\begin{equation*}
\min_{1\le k\le k_n}|\beta_{0k}|/\lambda_n\rightarrow \infty.
\end{equation*}

\end{enumerate}

Condition (A) is standard for local polynomial regression for estimation of curves with its derivatives. Noncompactly supported kernel can be used with increased technical complication. Condition (B) ensures that the parametric $\sqrt{n}$ convergence rate is achieved for integral functionals. The main purpose of using a weight function $w$ is to address undesirable boundary effect in local polynomial regression and to make the proof cleaner. Conditions $(G)-(I)$ are used to ensure the consistency and the oracle property of the final estimates which is standard in the high-dimensional regression literature.

\begin{thm}\label{th:an}
(Asymptotic normality of integral functional estimates) Suppose that conditions (A)-(E) holds. Then $\gamma_n^TG_n^{-1/2}(vech(\hat{M})-vech(M)\stackrel{d}{\rightarrow}N(0,1)$ for any $p_n(p_n+3)/2$ dimensional vector $\gamma_n$ with $||\gamma_n||=1$, where $\stackrel{d}{\rightarrow}$ means convergence in distribution and $G_n$ is the $\frac{p_n(p_n+3)}{2}\times \frac{p_n(p_n+3)}{2}$ asymptotic covariance matrix of $vech(\hat{M})$ which can be obtained from Lemma \ref{lem:bv} in section 5. 
\end{thm}

As presented in Lemma \ref{lem:bv} in section 5, the entries of the covariance matrix $G_n$ is of order $O(1/n)$, thus the rate of convergence for $vech(\hat{M})$ is $\sqrt{n}$. We note that we intentionally presented only the much simplified version of asymptotic normality. When functions $m_j, j=1,\ldots,p_n$ are not smooth enough, or the kernel bandwidth is chosen differently, or a lower order polynomial is used in nonparametric regression, it is possible to obtain asymptotic normality with slower rates, or with nonvanishing asymptotic bias. When this is the case, the following asymptotic results for $\hat{\beta}$ should be modified accordingly. In particular, the convergence rate of $\hat{\beta}$ depends critically on the convergence rate of $vech({\hat{M}})$. In Theorem \ref{th:consistency} we state the existence of a local minimizer in a neighborhood of the true parameter. Consistency of the global minimizer can be proved using peeling device as demonstrated in \cite{huangma06,huangma08}.

\begin{thm}\label{th:consistency}
(Local consistency) There exists a local minimizer $\hat{\beta}$ of $S(\beta)$ such that $||\hat{\beta}-\beta_0||=O_p(\frac{p_n^2}{\sqrt{n}\rho_{1n}})$, when conditions (A)-(H) holds.
\end{thm}

\begin{thm}\label{th:oracle}
(Oracle property) Let $\hat{\beta}=(\hat{\beta}_1^T,\hat{\beta}_2^T)^T$ be the local minimizer as stated in Theorem \ref{th:consistency}, where $\hat{\beta}_1$ is $k_n$ dimensional and $\hat{\beta}_2$ is $s_n$ dimensional. Under conditions (A)-(I), we have
\begin{itemize}
\item[(i)] $\hat{\beta}_2^T=0$ with probability converging to $1$.
\item[(ii)] For any $p_n$ dimensional vector $\gamma_n$ with unit norm, 
	\begin{equation*}
		\gamma_n^TP_n^{-1/2}(\hat{\beta_1}-\beta_{10})\rightarrow N(0,1)
	\end{equation*}
where 
\begin{equation*}
P_n=[(1,-\beta_{10}^T)\otimes Q^{-1}]\,\Phi_{p_n} G_n\Phi^T_{p_n} \,[(1,-\beta_{10}^T)\otimes Q^{-1}]^T
\end{equation*}
\end{itemize}
\end{thm}

The above theorem states that when $n$ is large, the zero coefficients are estimated as zero with high probability. The asymptotic distribution of the nonzero coefficients is the same as when the zero coefficients are known in advance, if the same nonparametric estimates are used in the first step. This fact can be seen easily from the proof in section 5 since the proof follows roughly the same lines whether or not the zero coefficients are known.  Note our oracle property is conditioned on the estimates of the solutions from the first step, which is not as clean as the oracle property stated in \cite{fan01}, for example.

\section{Numerical examples}
\subsection{Simulation}
First we illustrate some of the computational properties of our estimates with a simulation study. The functions $m$ are generated as follows. For $n=50, 100, 200$ time points, we use an even number of variables
$p_n=2r_n$ with $r_n=[n^{1/2}]$ and $[\cdot]$ denotes the integer part of a number. Note that asymptotically, $p_n$ used in the simulation does not tally with the assumptions used for theoretical investigations. The observation time points are $\{1/n,2/n,\ldots,1\}$. The $p_n\times
p_n$ coefficient matrix $A$ is generated as follows.
\[A_{2i-1,2i-1}=A_{2i,2i}=a_i, A_{2i-1,2i}=-A_{2i,2i-1}=b_i, i=1,\ldots,r_n,\]
\[A_{i,j}=0 \mbox{ all other } i,j,\]
\[a_i\stackrel{iid}{\sim} Uniform(-4,0), b_i\stackrel{iid}{\sim} Uniform(-10,10).\]
The structure of $A$ has the form
\[A=\left[\begin{array}{ccccccc}
a_1&b_1&0&0&\cdots&&\\
-b_1&a_1&0&0&\cdots&&\\
0&0&a_2&b_2&\cdots&&\\
0&0&-b_2&a_2&\cdots&&\\
\vdots&\vdots&\vdots&\vdots&\ddots&&\vdots\\
&&&&a_{r_n}&b_{r_n}\\
&&&&-b_{r_n}&a_{r_n}
\end{array}\right]\]
and the system of differential equations is written in matrix form as
\begin{equation}\label{eqn:sim}
m'=Am
\end{equation}
with $m(x)=(m_1(x),\ldots,m_{p_n}(x))^T$. 

We generate $m(0)$ from the uniform
distribution and solve the initial value differential equation
problem using the simple Euler's method. The solution is evaluated
at those $n$ time points and independent normal noise with
standard deviation $\sigma=0.1$ is added at each time point.

By the data generation mechanism, the evolution of one variable only
depends on the value of itself as well as one other
variable. The coefficient values $a_i$ are chosen to be negative so
that the solution of the differential equations is asymptotically
stable to avoid numerical problems.

In our experiment, we use $100$ samples  for each $n=50, 100, 200$ generated from  model (\ref{eqn:sim}). The two step estimator with SCAD penalty is applied. We use standard cross-validation to choose the bandwidth which leads to good empirical results. Cross-validation is also used in the second step. We compare its performance with another regression procedure in which one directly uses the noisy observations as covariates and finite differences as derivatives. That is,
the model is fitted by solving the following problem (showing only the first equation of the linear system)
\begin{eqnarray}\label{eqn:ts}
\hat{\beta}_{TS}=\arg\min_\beta \sum_{i=2}^n (\frac{Y_{i1}-Y_{i-1,1}}{X_i-X_{i-1}}-\sum_j \beta_j Y_{ij})^2+\sum_j p_{\lambda_n}(|\beta_j|).
\end{eqnarray}
Since this is similar to a discrete time series model, we call its minimizer the TS estimator. The results are shown in Table \ref{tab:sim}, where we consider several estimators: two-step estimator with SCAD penalty (SCAD), two-step estimator with no penalty (OLS), two-step estimator using only the two variables with nonzero coefficients (ORACLE),  (\ref{eqn:ts}) with SCAD penalty (TS-SCAD),  (\ref{eqn:ts}) using only the two covariates with nonzero coefficients (TS-ORACLE). The average mean squared errors (AMSE) shown in the table are the errors for the two nonzero coefficients only. We also show the average number of nonzero coefficients for the SCAD estimator. Estimation using  (\ref{eqn:ts}) produces much worse results compared to the two-step estimator which reduces the noise contained in the observed solutions.  We also see that the number of nonzero coefficients selected is close to the true value.

\begin{table}
\caption{AMSE for the simulation study. Numbers inside the brackets are the corresponding standard errors}
\label{tab:sim}
\begin{tabular}{lcccccc}
\hline
n & SCAD&OLS&ORACLE&TS-SCAD&TS-ORACLE&Average number of \\
				&&&&&&	nonzero coefficients \\
\hline
50 & 2.7 (0.44) &5.3 (1.42) & 2.4 (0.41) &6.2 (2.8)&7.5 (2.5) & 2.5\\
100& 2.6 (0.46) & 7.8 (0.77)& 2.4 (0.73)&8.2 (2.4)&12.7 (4.3) & 2.3\\
200& 2.9 (0.53) & 14.5 (3.2)& 2.1 (0.68)&20.7 (5.2)&13.0 (3.9)& 2.9\\
\hline
\end{tabular}
\end{table}

\subsection{Real data example}
Statistical inference of genetic regulatory networks is essential for understanding temporal interactions of regulatory elements inside the cells. For inferences of large networks, identification
of network structure is typically achieved under the assumption of sparsity of the networks. The increasing amount of high-throughput time course data has
provided biologists a window to the understanding of the
biomolecular mechanism of different species. The expression of
genes in these studies are indicative of the dynamic activities
occurring inside the organism. Such regulatory activities involve
complicated temporal interactions among different gene products,
forming genetic networks indicating the causal relationships
between different elements. 

We demonstrate the performance of the our penalized functional
model with the application to the cell cycle regulatory network of
Saccharomyces cerevisiae. The dataset comes from \cite{spellman98} which
provides a comprehensive list of cell cycle regulated genes
identified by time course expression analysis. We use the 24 unequally spaced time
points of the cdc15 synchronized expression data. Same as \cite{nam07}, we consider 20
genes including 4 transcription factors known to be involved in
regulatory functions during different stages of the cell cycle.

We apply our approach to this dataset. Only the part of the coefficent matrix showing the interactions
between each of the 20 genes and four transcription factor is presented in Table
\ref{tab:cycle}, and we compare the result with known
interactions retrieved from the YEASTRACT database \cite{teixeira06} and treat those as the background truth. For
this submatrix, we get \emph{PPV (positive predictive value)}=0.54 and \emph{sensitivity}=0.83. Since all statistical models are merely mathematical approximations to the
true world, it is plausible that automatically chosen model
undersmoothes the coefficients matrix to provide a better fit to
the data. One can also manually specify the smoothing
parameter in place of cross-validation to achieve desired sparsity of the networks.

\begin{table}[!t]
\caption{The reconstructed network structure with PPV=0.54 and Sensitivity=0.83. The interactions retrieved from database are denoted by '$\Box$' and the interactions inferred by the model are denoted by '$\times$'. \label{tab:cycle}}
\center{\begin{tabular}{l|cccc}\hline
      &    ace2     &    fkh1     &    swi4     &      swi5\\\hline
ace2  & $\boxtimes$ & $\boxtimes$ &             &              \\
fkh1  &             & $\boxtimes$ &             &              \\
swi4  &             & $\times$    & $\boxtimes$ &              \\
swi5  & $\times$    & $\boxtimes$ & $\times$    &   $\boxtimes$\\
sic1  & $\boxtimes$ & $\times   $ &             &   $\Box$     \\

cln3  &             &             & $\Box$      &   $\boxtimes$     \\
far1  &             &             &             &   $\times$   \\
cln2  &             &             & $\boxtimes$ &              \\
cln1  & $\times$    & $\Box$      & $\boxtimes$ &              \\
clb6  & 	    &             & $\boxtimes$ &   	       \\

clb5  &             &             & $\boxtimes$ &   $\times$   \\
gin4  &             &             & $\boxtimes$ &   $\times$   \\
swe1  &             &             & $\boxtimes$ &              \\
clb4  &             & $\boxtimes$ &             &              \\
clb2  & $\times$    & $\boxtimes$ & $\Box$      &              \\

clb1  &             & $\boxtimes$ & $\boxtimes$ &              \\
tem1  & 	    & $\boxtimes$ & $\times$    &   $\times$   \\
apc1  &             & $\times$    & $\times$    &   $\times$   \\
spo12 &             &             &             &   $\times$   \\
cdc20 & $\times$    & $\boxtimes$ &             &   $\times$   \\
\hline
\end{tabular}}{}
\end{table}

\section{Conclusions}
In this paper we studied the asymptotic properties of the two-step estimator in high-dimensional linear differential equations, when the size of the linear system diverges with the density of observed time points. Using local polynomial estimates in the first step combined with penalized regression in the second step, we have shown that the estimators correctly identify zero coefficients with probability converging to one and the estimators with nonzero coefficients are asymptotically normal with parametric convergence rates.  Since the covariates are observed with noise, the situation is similar to the errors-in-variables model where pretending the true covariates to be known will lead to estimators that are not even consistent. Thus it is crucial that only the smoothed solutions are plugged into the least square problem.

The most severe theoretical restriction comes from assumption (E). With the choice of $h=O(n^{-1/5})$ for example, the assumption $\sqrt{n}{p_n}(h^3+\frac{1}{nh^2})\rightarrow 0$ imposes the condition $p_n=o(n^{-1/10})$. This condition is used in the proof of Theorem \ref{th:an} to show the asymptotic normality of integral functionals. We suspect that this condition can be relaxed with more careful calculation.

Although we only focus on the case where the $\sqrt{n}$ rate is achieved, this of course depends on the smoothness of the solution as well as choice of bandwidth. In other situations, it might happen that the integral functionals converge with a different rate, which will also slow down the rate of the linear coefficient estimates. A comprehensive treatment considering all possible cases is beyond the scope of the current paper.

\section{Proofs}
First we investigate the asymptotic properties of the integral functionals $\int m_1'(x)m_i(x)w(x)dx$ and $\int m_i(x)m_j(x)w(x)dx$. The following Lemmas give asymptotic bias and variance of the local polynomial estimators and their proofs are similar to those found in \cite{huangfan99}, which only studied the quadratic functional $\int [m_i^{(d)}(x)]^2w(x)dx$.

We need to introduce more notations before presenting the lemmas. Let $A_1$ be the $n\times n$ matrix with the $(i,j)$-entry $a_{1ij}=\int W_0(\frac{X_i-x}{h})W_1(\frac{X_j-x}{h})w(x)dx$. Similarly, let $A_2$ be the $n\times n$ matrix with the $(i,j)$-entry $a_{2ij}=\int W_0(\frac{X_i-x}{h})W_1(\frac{X_j-x}{h})w(x)dx$. The matrix $A_2$ is symmetric while $A_1$ is not. Let $B_1$ denote the symmetrized version of $A_1$, i.e., $B_1=(A_1+A_1^T)/2$. 

\begin{lemma}\label{lem:trace}
Under conditions (A)-(C) together with $nh\rightarrow \infty$, conditioning on the random time points $\{X_i\}_{i=1}^n$,
\begin{eqnarray*}
tr(A_1)&=&(C+o_p(1))\frac{1}{nh^2}\\
tr(A_1^TA_1)&=&(C+o_p(1))\frac{1}{n^2h^3}\\
tr(A_2)&=&(C+o_p(1))\frac{1}{nh}\\
tr(A_1^TA_2)&=&(C+o_p(1))\frac{1}{n^2h^2}\\
tr(A_2^2)&=&(C+o_p(1))\frac{1}{n^2h}
\end{eqnarray*}
where in the above expressions, different appearances of $C$ denotes different constants depending on $K$ and $w$. Similar observation applies to the next lemma as well.
\end{lemma}
\begin{proof}
The calculations for $tr(A_2)$ and $tr(A_2^2)$ are special cases studied in \cite{huangfan99}. In particular, the calculations are contained in the proof of their Theorem 4.1, equations (7.3) and (7.19). The proofs for all other cases are similar and omitted. 
\end{proof}
\begin{lemma}\label{lem:bv}
Let $\theta_1=\int m_1'm_1w$ and $\theta_2=\int m_1'm_2w$, $\theta_3=\int m_1m_1w$ and $\theta_4=\int m_1m_2w$. Use $\hat{\theta}_1, \ldots, \hat{\theta}_4$ for the corresponding  estimated version with true functions replaced with local polynomial estimates. Under the same conditions as stated for Lemma \ref{lem:trace}, 
\begin{itemize} 
    \item[(a)] the asymptotic bias is
	\begin{eqnarray*}
	E(\hat{\theta}_1)-\theta_1&=&(C_1+o_p(1))h^3+(C_2+o_p(1))\frac{1}{nh^2}\\
	E(\hat{\theta}_2)-\theta_2&=&(C_1+o_p(1))h^3\\
	E(\hat{\theta}_3)-\theta_3&=&(C_1+o_p(1))h^4+(C_2+o_p(1))\frac{1}{nh}\\
	E(\hat{\theta}_4)-\theta_4&=&(C_1+o_p(1))h^4\\
	\end{eqnarray*}
    \item[(b)] the asymptotic variance is
	\begin{eqnarray*}
	Var(\hat{\theta}_1)&=&(C_1+o_p(1))\frac{1}{n^2h^3}+(C_2+o_p(1))\frac{1}{n}\\
	Var(\hat{\theta}_2)&=&(C_1+o_p(1))\frac{1}{n^2h^3}+(C_2+o_p(1))\frac{1}{n}\\	
	Var(\hat{\theta}_3)&=&(C_1+o_p(1))\frac{1}{n^2h}+(C_2+o_p(1))\frac{1}{n}\\
	Var(\hat{\theta}_4)&=&(C_1+o_p(1))\frac{1}{n^2h}+(C_2+o_p(1))\frac{1}{n}\\
	\end{eqnarray*}
	
    \item[(c)] similarly, we can calculate the covariances. For example, 
	\begin{equation*}
	Cov(\hat{\theta}_1,\hat{\theta}_2)=(C_1+o_p(1))\frac{1}{n^2h^3}+(C_2+o_p(1))\frac{1}{n}
	\end{equation*}
\end{itemize}
The above biases and variances are implicitly conditioned on the random time points $\{X_i\}_{i=1}^n$.
\end{lemma}

\begin{proof}
The calculation follows that of \cite{huangfan99} and we refer the reader to that paper for details,  giving here only some short explanations of the proof as well as pointing out the differences when dealing with non-quadratic forms which were not studied in \cite{huangfan99}. For ease of notation, within the current Lemma, we let $Y^T=(Y_1,\ldots,Y_n),$ and $Z^T=(Z_1,\ldots,Z_n)$ be the noisy observations for functions $m_1$ and $m_2$ respectively, with additive noise $\epsilon_1=(\epsilon_{11},\ldots,\epsilon_{n1})$ and $\epsilon_2=(\epsilon_{12},\ldots,\epsilon_{n2})$.

 Obviously we have $\hat{\theta}_1=Y^TB_1Y=m_1^TB_1m_1+2m_1B_1\epsilon_1+\epsilon_1^TB_1\epsilon_1$. Given $\{X_i\}$, the conditional expectation is $m_1^TB_1m_1+\sigma^2 tr(B_1)$. 
$m_1^TB_1m_1$ contributes to the $O_p(h^3)$ term in the bias and $tr(B_1)=O_p(\frac{1}{nh^2})$. For $\hat{\theta}_2=Y^TA_1Z=m_1A_1m_2+m_1A_1\epsilon_2+m_2^TA_1^T\epsilon_1+\epsilon_1A_1\epsilon_2$, since the noises $\epsilon_1$, $\epsilon_2$ are independent, the bias only comes from $m_1^TB_1m_2=O_p(h^3)$.

The conditional variance of $Y^TB_1Y$ is $4\sigma^2m_1B_1^2m_1+2\sigma^4 tr(B_1^2)$, where the first term comes from the variance of $2m_1B_1\epsilon_1$ and the second term comes from variance of $\epsilon_1B_1\epsilon_1$. Calculations show that the first term is of order $O_p(\frac{1}{n^2h^3}+\frac{1}{n})$ while the second term is $O_p(\frac{1}{n^2h^3})$. All other expressions are derived in the same way.

\end{proof}

\begin{proof}[Proof of Theorem \ref{th:an}]
We want to show the asymptotic normality of $vech(\hat{M})$. For convenience, we denote the components of $vech(\hat{M})$ by $\hat{u}_k, k=1,2,\ldots, p_n(p_n+3)/2$ and the components of $vech(M)$ by $u_k$. From the proof of Lemma \ref{lem:bv}, one can see that $\hat{u}_k$ is of the form
$\hat{u}_k=(m_i+\epsilon_i)^TC(m_j+\epsilon_j)$ with either $C=A_1$ or $C=A_2$ and $i$ possibly equals $j$. Thus $\hat{u}_k=m_iCm_j+m_i^TC\epsilon_j+m_j^TC^T\epsilon_i+\epsilon_i^TC\epsilon_j$. By the proof of the lemma, the first term is $u_k+O_p(h^3+\frac{1}{nh^2})$ and the last term is $O_p(\sqrt{(tr(C^2)})=O_p(\sqrt{\frac{1}{n^2h^3}})$. This is the  pseudo-quadratic situation as described in \cite{huangfan99} since the linear term, which is of order $O_p(\frac{1}{\sqrt{n}})$, dominates the quadratic term. Thus we have $vech(\hat{M})-vech(M)=V+R$, where each component of $V$ is a linear combination of the errors and of order $\frac{1}{\sqrt{n}}$ due to the calculation of Lemma \ref{lem:bv}, and each components of $R$ is of order $O_p(h^3+\frac{1}{nh^2})$. Let the asymptotic covariance matrix of $vech(\hat{M})-vech(M)$ be denoted by $G_n$, which is the dominating term in Lemma \ref{lem:bv}\,(b). Obviously  $\gamma_nG_n^{-1/2}V$ has an asymptotic standard normal distribution. The same will be true for $\gamma_nG_n^{-1/2}(vech(\hat{M})-vech(M))$ if $\gamma_nG_n^{-1/2}R=o_p(1)$. Since $||G_n^{-1/2}||$ is of oder $O_p(\sqrt{n})$, we have $\gamma_nG_n^{-1/2}R=O_p(\sqrt{n}p_n(h^3+\frac{1}{nh^2}))$. The results follow from assumption (E). 
\end{proof}

The following lemma bounds the eigenvalues of $\hat{Q}$.
\begin{lemma}\label{lem:eigen}
If $p_n/\sqrt{n}=o(\rho_{1n})$ and $p_n/\sqrt{n}=o(\rho_{2n})$, then $\lambda_{\max}(\hat{Q})/\rho_{2n}=O_p(1)$ and $\rho_{1n}/\lambda_{\min}(\hat{Q})=O_p(1)$.
\end{lemma}
\begin{proof}
By the Gershgorin Circle Theorem (\cite{golub96} Theorem 7.2.1),   the eigenvalues of $\hat{Q}-Q$ lie inside the interval $[-C\frac{p_n}{\sqrt{n}},+C\frac{p_n}{\sqrt{n}}]$ with high probability for large enough constant $C$. By the assumption, we have $\lambda_{\min}(\hat{Q})\ge \lambda_{\min}(Q)-C\frac{p_n}{\sqrt{n}}\ge \rho_{1n}/2$. Same proof applies to for $\lambda_{\max}(\hat{Q})$.

\end{proof}

\begin{proof}[Proof of Theorem \ref{th:consistency}]

Let $\tau_n=C\frac{p_n^2}{\sqrt{n}\rho_{1n}}$. Following \cite{fan04}, we will show that for any $\epsilon>0$ we can find a large enough constant $C$ such that
\begin{equation}\label{eqn:consistency}
P\{\sup_{||u||=1} S(\beta_0+\tau_n u)>S(\beta_0)\}\ge 1-\epsilon.
\end{equation}
Simple calculations show that
\begin{eqnarray}
S(\beta_0+\tau_n u)-S(\beta_0)&\ge&-2\tau_nu^T(\hat{b}-\hat{Q}\beta_0)+\tau_n^2u^T\hat{Q}u\label{eqn:diff}\\
&&+\sum_{j=1}^{k_n} p_{\lambda_n}(\beta_{0j}+\tau_nu_j)-\sum_{j=1}^{k_n}p_{\lambda_n}(\beta_{0j}).\nonumber
\end{eqnarray}
Since $\beta_0$ minimizes $\int (m_1'(x)-\beta^Tm(x))^2w(x)dx$, we have the normal equation $Q\beta_0=b$.  Thus we can bound the first term on the right hand side of (\ref{eqn:diff}) as
\begin{eqnarray*}
|2\tau_nu^T(\hat{b}-\hat{Q}\beta_0)|&=&|2\tau_nu^T(\hat{b}-b-(\hat{Q}-Q)\beta_0)|\\
&=&O_p(\tau_n\frac{p_n^2}{\sqrt{n}}).
\end{eqnarray*}
By Lemma \ref{lem:eigen}, the second term $\tau_n^2u^T\hat{Q}u$ can be bounded below by $O_p(\tau_n^2\rho_{1n})$. Thus the second term dominates the first term  when $C$ is large enough.
Same as (5.5) and (5.6) in \cite{fan04} the contribution of the penalty terms are also dominated and (\ref{eqn:consistency}) is proved.
\end{proof}

To make the proof of Theorem \ref{th:oracle} less cluttered, we show the variable selection consistency in a separate Lemma.
\begin{lemma} \label{lem:vsconsistency}
Under the assumptions of Theorem \ref{th:oracle}, the local minimizer found in Theorem \ref{th:consistency} with $\hat{\beta}=(\hat{\beta}_1^T,\hat{\beta}_2^T)^T$ is consistent in variable selection: $\hat{\beta}_2=0$ with probability converging to one.
\end{lemma}
\begin{proof}
For the objective function $S(\beta)$ defined in (\ref{eqn:obj}), we have
\begin{equation}\label{eqn:der}
\frac{\partial S}{\partial\beta_j}=2(\hat{Q}\beta-\hat{b})+p_{\lambda_n}'(|\beta_j|)sgn(\beta_j).
\end{equation}
Consider $j=k_n+1,\ldots,p_n$. When $0<|\beta_j|<C\frac{p_n^2}{\sqrt{n}\rho_{1n}}$, we can bound
\begin{eqnarray*}
||\hat{Q}\beta-\hat{b}||&=&||\hat{Q}\beta_0-\hat{b}+\hat{Q}(\beta-\beta_0)||\\
&=&||(\hat{Q}-Q)\beta_0-(\hat{b}-b)+\hat{Q}(\beta-\beta_0)||\\
&\le&||\hat{Q}-Q||\,||\beta_0||+||\hat{b}-b||+||\hat{Q}||\,||\beta-\beta_0||\\
&\le&\frac{p_n^2}{\sqrt{n}}+\frac{p_n}{\sqrt{n}}+\rho_{2n}\frac{p_n^2}{\sqrt{n}\rho_{1n}}=O_p(\frac{\rho_{2n}}{\rho_{1n}}\frac{p_n^2}{\sqrt{n}}).\\
\end{eqnarray*}
Since $\frac{p_n^2}{\sqrt{n}\rho_{1n}\lambda_n}\rightarrow 0$, and $|\beta_j|\le C \frac{p_n^2}{\sqrt{n}\rho_{1n}}$, we have $\lim\inf p'_{\lambda_n}(|\beta_j|)/\lambda_n>0$ by the form of the SCAD penalty. So the penalty term in (\ref{eqn:der}) dominates. Thus 
\begin{eqnarray*}
\frac{\partial S}{\partial\beta_j}>0 ~~~~\mbox{for}~ 0<\beta_j<C\frac{p_n^2}{\sqrt{n}\rho_{1n}}\\
\frac{\partial S}{\partial\beta_j}<0 ~~~~\mbox{for}~ 0>\beta_j>-C\frac{p_n^2}{\sqrt{n}\rho_{1n}}\\
\end{eqnarray*}
which implies that the minimum is achieved at exactly $\beta_j=0$.
\end{proof}

\begin{proof}[Proof of Theorem \ref{th:oracle}]
Now that part (i) has been proved in Lemma \ref{lem:vsconsistency}, we focus on asymptotic normality. Since $\hat{\beta}_2=0$ with probability converging to one, we can concentrate on $\hat{\beta}_1$ and denote it simply as $\hat{\beta}$. First, by assumptions (G), (I) and Theorem \ref{th:consistency}, $\nabla p_{\lambda_n}(\hat{\beta}_1)=0$. From the fact that $\nabla S(\hat{\beta})=0$, it follows that
\begin{equation*}
-\hat{b}+\hat{Q}\hat{\beta}=0.
\end{equation*}
Together with $-b+Q\beta_0=0$, we have
\begin{equation*}
-(\hat{b}-b)+(\hat{Q}-Q)\beta_0+\hat{Q}(\hat{\beta}-\beta_0)=0,
\end{equation*}
or
\begin{eqnarray*}
\hat{\beta}-\beta_0&=&Q^{-1}[(\hat{b}-b)-(\hat{Q}-Q)\beta_0]+\mbox{higher order terms}\\
&=&Q^{-1}(\hat{M}-M){1\choose -\beta_0}.
\end{eqnarray*}
Since
\begin{eqnarray*}
Q^{-1}(\hat{M}-M){1\choose -\beta_0}&=&vec(Q^{-1}(\hat{M}-M){1\choose -\beta_0})\\
&=&[(1,-\beta_0^T)\otimes Q^{-1}] vec(\hat{M}-M)\\
&=&[(1,-\beta_0^T)\otimes Q^{-1}] \Phi_{p_n} vech(\hat{M}-M)
\end{eqnarray*}
the conclusion follows directly from Theorem \ref{th:an}.
\end{proof}

\bibliographystyle{plain}
\bibliography{papers,books}

\end{document}